%% file: main.tex
\documentclass{amsart}

\input{preamble}

\begin{document}

\title{Note on a bomb dropped by Mr Conant and Mr Kruckman, and its consequences for the theory ACFG}


\author[C. d'Elb\'{e}e]{Christian d\textquoteright Elb\'ee$^\dagger$}
\address{School of Mathematics, University of Leeds\\
Office 10.17f LS2 9JT, Leeds}
\email{C.M.B.J.dElbee@leeds.ac.uk}
\urladdr{\href{http://choum.net/\textasciitilde chris/page\textunderscore perso/}{http://choum.net/\textasciitilde chris/page\textunderscore perso/}}

\date{\today}

\maketitle

\section*{Introduction}

In \cite{conant2023surprising} Conant and Kruckman dropped a bomb on the model theory community by exhibiting a counterexample (actually 3) to a folklore fact which can be stated as follows:
\[A\indi d _C B\implies A\indi d _C \acl(BC) \quad \quad (\mathghost)\]
where $\indi d$ is dividing independence.  In the setting of \cite{delbee2023axiomatic}, they prove that in general $\indi d$ fails right \ref{CLO}, see also Subsection \ref{preliminaries} for axioms of independence relations. From this, statements such as $\indi d\to \indi a ^M$, where $A \indi a ^M _C B$ iff $\acl(AD)\cap \acl(BC) = \acl(C)$ for all $C\seq D\seq \acl(BC)$,
are not longer true and this has great and unforeseen consequences on several other papers, see \cite[Section 5]{conant2023surprising}. Among the results that has to be revised with the failure of $(\mathghost)$ are several general description of $\indi d$ in strictly NSOP$_1$ theories. This is also the case for the authors paper \cite{delbee21acfg} on the model-companion of the expansion of a field of fixed positive characteristic by an additive subgroup, otherwise known as ACFG. In particular, it was proved in \cite{delbee21acfg} that $\indi f = \indi d$ and that those are obtained by `forcing base monotonicity' on Kim-independence (see below for more details). In this note, we start in Section \ref{sec:ACFG} by giving two counterexample to $\indi f = \indi d$ in ACFG. The first one (Subsection \ref{subsection:ex1}) is really happening in the home sort (using in particular that algebraically closed fields have build-in codes for finite sets). The second one (Subsection \ref{subsection:ex2}) is closer to the counterexample \cite[Section 3.2]{conant2023surprising} and uses that models $(K,G)$ of ACFG interpret a generic (symmetric) binary map, namely $(x,y)\mapsto \pi(xy)$ where $\pi:K\to K/G$ is the canonical projection.

In Section \ref{sec:axiomaticnaivemon}, we expand on the subtleties around `forcing base monotonicity' and how this can be done with or without forcing the \ref{CLO} axiom. Define the following extensions of a given independence relation $\ind$ (essentially due to Adler \cite{adlerthesis,Adl08}):
\begin{align*}
    A\indi M _C B &:\iff A\ind_D B\text{ for all $C\seq D\seq \ac(BC)$}.\\
A\indi m _C B &:\iff A\ind_D B\text{ for all $C\seq D\seq BC$}.
\end{align*}

Until now, the `monotonisation' usually referred to $\indi M$ and was sometimes denoted $\indi m$ (sic), but we fix the notations above, following Adler and \cite{conant2023surprising}. In this note, $\indi m$ will be referred to as the `naive' monotonisation. \\

\subsection*{Caveat} In \cite{delbee21acfg}, there is a mistake in Definition 4.1, where the monotonisation $\indi m$ is defined as above but the intended definition and the definition used in the proofs of \cite{delbee21acfg} really is the one of $\indi M$. This is a notationnaly fortunate mistake but makes the problems associated to the failure of $(\mathghost)$ even more confused and this note more necessary. The reader should read every mention of `$\indi m, \indi w ^m$\ ' in \cite{delbee21acfg} as `$\indi M, \indi w ^M$\ '. Note also that $\indi M$ in Adler or in \cite{conant2023surprising} is $(\indi a)^M$ for us, where $A\indi a _C B$ if and only if $\acl(AC)\cap \acl(BC) = \acl(C)$. \\

Here is a list of updates on the consequences of the failure of $(\mathghost)$ in ACFG.
\begin{itemize}
    \item We still have $\indi f = \indi K ^M$, where $\indi K$ is Kim-independence, although the statement \cite[Proposition 4.14]{delbee21acfg} is now false. We expand on this just below.
    \item We have $\indi{da} = \indi f$, where $\indi {da}\ $ is defined below. This is similar to \cite[Subsection 5.3]{conant2023surprising}.
    \item We have $\indi d\to (\indi a )^M$, simply because we already have $\indi K\to (\indi a )^M$, see Remark \ref{rk:pregeommon}.
\end{itemize}

One of the main consequences of the failure of $(\mathghost)$ is that if $\indi d\to \ind$, then one does not necessarily have $\indi d\to \indi M$ but only $\indi d\to \indi m$. This simple observation make \cite[Proposition 4.14]{delbee21acfg} no longer true and thus the proof of $\indi f = \indi d$ obsolete. In Section \ref{sec:axiomaticnaivemon} we give a correct account of \cite[Proposition 4.14]{delbee21acfg}, Corollary \ref{cor:correctversionofprop4.14} below which is essentially \cite[Proposition 4.14]{delbee21acfg} but replacing $\indi M$ by $\indi m$. Unfortunately, Corollary \ref{cor:correctversionofprop4.14} cannot be applied in ACFG to conclude that $\indi d = \indi f$. Another variant of a correct statement of \cite[Proposition 4.14]{delbee21acfg} can be obtained with the following definition (\cite[Subsection 2.4]{conant2023surprising}):
    \[A\indi{da}_C B:\iff A\indi d_C \acl(BC).\]
    Now, $\indi d\to \ind$ implies $\indi{da}\to \indi M$. By reading the proof of \cite[Proposition 4.14]{delbee21acfg}, we obtain:\\

    \noindent \textbf{Correct version of \cite[Proposition 4.14]{delbee21acfg}.} Let $\ind$ be an invariant relation satisfying left and right \ref{MON} such that
    \begin{enumerate}
        \item $\indi{da} \to \ind$
        \item $\ind$ satisfies $\indi h$-amalgamation over models
        \item $\indi M$ satisfies \ref{EXT}
    \end{enumerate}
    then $\indi M = \indi f = \indi{da}$.\\

In ACFG, $(\indi K )^M$ satisfies \ref{EXT} (\cite[Corollary 4.18]{delbee21acfg}), and $\indi K$ satisfies $(1)$ and $(2)$ hence we conclude:
\[\indi f = \indi{da} = (\indi K )^M.\]

Recall the following notation from Adler, given an independence relation $\ind$:
\[A\indi * _C B \iff \text{for all $D\supseteq B$, there exists $A'\equiv_{BC} A$ with $A'\ind_C D$}.\]
It is in the current folklore on NSOP$_1$ theories that $\indi K ^{M*} = \indi f$. We observe (Proposition \ref{prop:twomonsameextenNSOP1}) that either version of the monotonisation are enough to get this result:
\[(\indi K ) ^{M*} = (\indi K )^{m*} = \indi f\]
In particular in ACFG we have $\indi K ^M = \indi K^{m^*}\neq \indi K ^m$. The picture of independence relations in ACFG can be summarize as follows:
    \begin{center}
        \begin{tikzcd}[column sep=small]
           \indi K ^{m*} \ar[equal]{d}  &    &  &  &  & \\
    \indi K ^M \ar[equal]{d} &    &  &  &  & \\
    \indi f  \ar{r} &  \indi d \ar{r}{=?} & \indi K ^m \ar{r} & \indi K \ar{r}  & \indi a ^M \ar{r} & \indi a\\
    \indi{da} \ar[equal]{u} &  & & & &
\end{tikzcd}
    \end{center}
where every arrow is strict, except maybe the one between $\indi d$ and $\indi K ^m$, for which we do not know.

\subsection*{Notations} We use $A^\alg$ for the field theoretic algebraic closure and $\indi \alg\ $ for the algebraic independence relation in the sense of fields. We use $A\indi h_C B$ if and only if $\tp(A/BC)$ is a coheir of $\tp(A/C)$.
\clearpage

\section{Forking and dividing differ for types}\label{sec:ACFG}

\subsection{Preliminaries}
Fix a prime number $p$. Let $\LL_G$ be the expansion of the language of rings by a unary predicate $G$. Let $T_G = T_G(p)$ be the $\LL_G$-theory of fields of characteristic $p$ in which $G$ is an additive subgroup. By \cite{delbee21acfg,dE21} $T_G$ admits a model-companion, denoted ACFG, which is NSOP$_1$ and not simple. We recall some facts about the theory ACFG, see \cite{delbee21acfg,dE21} for details.

\begin{fact}\label{fact:factACFG}
    Let $(K,G)$ be a sufficiently saturated model of ACFG, where $G$ is the generic subgroup of $(K,+)$. Let $A,A',B,C$ be small subsets of $K$. We denote by $G(A)$ the set  $G\cap A$.
    \begin{enumerate}
        \item $A\equiv_C A'$ if and only if there is an $\LL_G$-isomorphism $f$ over $C$ between the two $\LL_G$-structures \[((AC)^\alg,G((AC)^\alg)\cong ((A'C)^\alg, G((A'C)^\alg))\]
        which maps $A$ to $A'$.
        \item $A\indi K _C B$ if and only if $A\indi \alg_C B$ and $
        G((AC)^\alg+(BC)^\alg) = G((AC)^\alg)+G((BC)^\alg)$.
        \item $A\indi f_C B$ if and only if $A\indi K _D B$ for all $C\seq D\seq \acl(BC)$.
        
    \end{enumerate}
\end{fact}
See \cite[Subsection 2.1]{delbee21acfg} for an explicit description on how to prove that types are consistent in ACFG.

\subsection{$\indi d\neq \indi f$ in the theory ACFG}\label{subsection:ex1}
Fix a model $M$ of ACFG and a bigger model $(K,G)$ extending $M$. Let $d_1,d_2\in K$ be algebraically independent over $M$ with $G((Md_1d_2)^\alg) = G(M)$. Let $b = (d_1+d_2,d_1d_2)\in K^2$, it is classical that the tuple $b$ is a canonical parameter for the set $\set{d_1,d_2}$, in particular, $d_1,d_2\in (\F_p(b))^\alg$ hence $(Md_1d_2)^\alg = (Mb)^\alg$.
Let $a\in K$ be an element such that 
\[G((Mab)^\alg)) = G(M)\oplus \Span_{\F_p}(ad_2-d_1). \quad (\ast)\]
We leave to the reader to check that $G((Ma)^\alg) = G((Md_1d_2)^\alg) = G((Mad_1)^\alg) = G((Mad_2)^\alg = G(M)$.
We claim that $a\indi d _M b$ whereas $a\nindi f _M b$. 

First, we check that $a\nindi f _M b$. Assume otherwise, then as $d_1,d_2\in (Mb)^\alg$, we have $a\indi K_{Md_2} d_1b$ using Fact \ref{fact:factACFG} (3). Then $ad_2 - d_1\in G((Mad_2)^\alg+(Mb)^\alg)$ but $ad_2 - d_1\notin G((Mad_2)^\alg)+G((Mb)^\alg = G(M)$, contradicting Fact \ref{fact:factACFG} (2).

To prove that $a\indi d_M b$, we need a small claim.

\begin{claim}\label{claim:indisceniblevectorspace}
    Let $(d^n)_{n<\omega}$ be an infinite sequence of pairs $d_n = (d_1^nd_2^n)$ which is indiscernible in some theory of $\F_p$-vector space. Then one (and only one) of the three following configurations occur:
    \begin{enumerate}
        \item $(d_1^n)_{n<\omega}$ and $(d_2^n)_{n<\omega}$ are both infinite and $\set{d_1^n,d_2^m\mid n,m<\omega}$ is $\F_p$-linearly independent
        \item $(d_1^n)_{n<\omega}$ is constant and $\set{d_2^n\mid n<\omega}$ is $\F_p$-linearly independent over $d_1^0$.
        \item $(d_2^n)_{n<\omega}$ is constant and $\set{d_1^n\mid n<\omega}$ is $\F_p$-linearly independent over $d_2^0$.
    \end{enumerate}
\end{claim}
\begin{proof}
    Assume that the first case does not occur. Then there exists finite sets $I,J\seq \omega$ and $(\lambda_i)_{i\in I}$, $(\mu_j)_{j\in J}\seq \F_p\setminus \set{0}$ such that
    \[\sum_I \lambda_i d_1^i+\sum_J \mu_j d_2^j = 0. \quad (\star)\]
    Let $i_0 = \max (I\cup J)$. Assume that $i_0\in I$. Then $d_1^{i_0}\in V := \Span_{\F_p}(d_1^i,d_2^j\mid i\in I\setminus\set{i_0},j\in J)$ and by indiscernibility the set $\set{d_1^{k}\mid k\geq i_0}$ is included in the finite set $V$, which implies that $d_1^n = d_1^m$ for all $n\neq m$. In particular we may assume that $I = \set{i_0}$. If $J\neq \emptyset$, let $j_0 = \max J$. Reasoning as above, $(\star)$ implies that $\set{d_2^k\mid k\geq j_0}$ is included in a translate of $W = \Span_{\F_p}(d_1^{i_0},d_2^j\mid j\in J\setminus\set{j_0})$ hence by indiscernibility $(d_2^n)_{n<\omega}$ is constant. This contradicts that $(d_1^n,d_2^n)_{n<\omega}$ is infinite hence $(2)$ holds.

    If $i_0\in J$, by a similar argument one conclude that $(3)$ holds.
\end{proof}

To prove that $a\indi d _C b$, let $(b_n)_{n<\omega}$ be any infinite indiscernible sequence in $\tp(b/M)$. For each $n$ there is $d_1^n,d_2^n$ such that $b_n = (d_1^n+d_2^n, d_1^nd_2^n)$ and we may assume that $(b_nd_1^nd_2^n)_{n<\omega}$ is indiscernible over $M$. Note that $(d_1^n,d_2^n)_{n<\omega}$ is infinite. By the claim there are three configurations which reduce to the two following two cases:
\begin{itemize}
    \item \textbf{Case I.} $(d_2^n)_{n<\omega}$ is infinite. This corresponds to configurations $(1)$ and $(2)$ from Claim \ref{claim:indisceniblevectorspace}.
    \item \textbf{Case II.} $(d_2^n)_{n<\omega}$ is constant, this is configuration $(3)$ of Claim \ref{claim:indisceniblevectorspace}
\end{itemize}

\textbf{Case I}. Assume that we are in the configuration $(2)$, i.e. $(d_1^n)_{n<\omega}$ is constant and $(d_2^n)_{n<\omega}$ is $\F_p$-independent over $d_1^0$. We now prove that there exists $a'$ algebraically independent over $(M(b_n)_n)^\alg$ and such that
\[G((Ma'(b_n)_n)^\alg) = G(M) \oplus \bigoplus_{n<\omega}\Span_{\F_p}(a'd_2^n-d_1^0).\]
Assume that such $a'$ exists, then it is an easy exercise to check that $G((Ma'b_n)^\alg) = G(M)\oplus\Span_{\F_p}(a'd_2^n-d_1^n)$ for each $n<\omega$. By Fact \ref{fact:factACFG} (1), we conclude $a'b_n\equiv_M ab$ for all $n$, hence $a\indi d_M b$.
In order to check that such $a'$ exists, it suffices to prove that if $x$ algebraically independent over $K$, then
\[[G(K) \oplus \bigoplus_{n<\omega}\Span_{\F_p}(xd_2^n-d_1^0)]\cap K = G(K)  \quad (\ast\ast)\]
and then use that $(K,G)$ is existentially closed in the $\LL_G$ structure $(K(x)^\alg, H)$ where $H = G(K) \oplus \bigoplus_{n<\omega}\Span_{\F_p}(xd_2^n-d_1^0))$. We check $(\ast\ast)$: assume that $g_K+\sum_n \lambda_n(xd_2^n-d_1^0)\in K$ for some $g_K\in G(K)$ and $\lambda_n\in \F_p$. Using that $x$ is transcendental over $K$, we get $\sum_n \lambda_nd_2^n =0$. As $(d_2^n)_n$ is linearly independent, we get $\lambda_n = 0$ for all $n$, hence $g_K+\sum_n \lambda_n(xd_2^n-d_1^0)\in G(K)$. The other inclusion is trivial.

Configuation $(1)$ is treated similarly, with $(\ast \ast)$ replaced by 
\[[G(K) \oplus \bigoplus_{n<\omega}\Span_{\F_p}(xd_2^n-d_1^n)]\cap K = G(K).\]

\textbf{Case II.} In this case one easily sees\footnote{To get such an $a'$, one would consider $x$ algebraically independent over $K$ and check that $[G(K) \oplus \bigoplus_{n<\omega}\Span_{\F_p}(xd_2^0-d_1^n)]\cap K = G(K)$. Take $g_K+\sum_n \lambda_n(xd_2^0-d_1^n)\in K$ with $g_K\in G(K)$. Then, as $x$ is transcendental over $K$, we get $(\sum_n \lambda_n)d_2^0 = d_1^n$ which has a nontrivial solution, e.g. if $\sum_n\lambda_n = 0$ and $\lambda_n\neq 0$. } that there is no $a'$ algebraically independent over $M(b_n)_n$ with 
\[G((Ma'(b_n)_n)^\alg) = G(M) \oplus \bigoplus_{n<\omega}\Span_{\F_p}(a'd_2^0-d_1^n).\]
However, interverting the roles of $d_1^n$ and $d_2^n$ and reasoning as in Case I, there is $a'$ such that 
\[G((Ma'(b_n)_n)^\alg) = G(M) \oplus \bigoplus_{n<\omega}\Span_{\F_p}(a'd_1^n-d_2^0)\]
Then $a'd_2^nd_1^n \equiv_M ad_1d_2$ for all $n<\omega$, which implies that $a'b^n\equiv_M ab$ for all $n<\omega$, so we conclude $a\indi d _M b$.

\subsection{$\indi d\neq \indi f$ using a generic binary function in ACFG}\label{subsection:ex2}

In ACFG$^\eq$, a different witness of $\indi d\neq \indi f$ can be described, which seems very similar to the counter-example to $\indi d\neq \indi f$ in the generic binary function.

Imaginaries in ACFG have been described in \cite[Section 3]{delbee21acfg}. Assume that $(K,G)$ is a model of ACFG and let $\pi : K \to K/G$ be the canonical projection and we consider the two sorted structure $(K,K/G,\pi)$ where $K$ has the full field structure and $K/G$ the group structure. We consider $(K,K/G,\pi)$ as a substructure of $(K,G)^\eq$. Then \cite[Theorem 3.15]{delbee21acfg} states that $(K,K/G,\pi)$ weakly eliminates imaginaries. 
The structure $(K,K/G,\pi)$ can be seen as a ``generic forgetting" structure where the map $\pi: K\to K/G$ is a generic map only preserving the additive group structure in the field $K$. Precisely because $\pi$ does not preserves the multiplicative structure of $K$, one sees that $(x,y)\mapsto \pi(xy)$ is a generic binary (and symmetric) map.

It is an easy exercise to prove that $A\indi K_C B$ if and only if $A\indi \alg _C B$ and $\pi((AC)^\alg)\cap \pi((BC)^\alg) = \pi(C^\alg)$. Then $\indi f$ is given by forcing base monotonicity and algebraic extension on $\indi K$.

The second counterexample to $\indi d = \indi f$ in ACFG$^\eq$ is the following. Start with an elementary substructure $(M,\pi(M))$ of $(K,K/G)$ and let $u_1,u_2,d_1,d_2,a$ be as follows:
\begin{itemize}
    \item $u_1,u_2\in K/G$ are $\F_p$-linearly independent over $\pi(M)$.
    \item $d_1,d_2\in K$ are algebraically independent over $M$ with $\pi(d_i) = u_i$ and $\pi((Md_i)^\alg = \pi(M)\oplus \Span_{\F_p}(u_i)$, for $i = 1,2$.
    \item $a$ is algebraically independent over $Md_1d_2$ and $\pi(ad_1) = 0$, $\pi(ad_2) = u_1$ and $\pi(Mad_1d_2) = \pi(M)\oplus \Span_{\F_p}(u_1,u_2)$.
\end{itemize} 
Then for $b = (d_1+d_2, d_1d_2)$ we have $a\nindi f_M b$ and $a\indi d_M b$. The proof is similar to the one above. Note however that the two examples differ greatly. In the first example we have $G((Mab)^\alg)) = G(M)\oplus \Span_{\F_p}(ad_2-d_1)$, so the group in $(Mab)^\alg$ is `small' but in the second example we have $\pi(Mad_1d_2) = \pi(M)\oplus \Span_{\F_p}(u_1,u_2)$ hence the group in $(Mab)^\alg$ is `big'.

\section{Naive monotonisation}\label{sec:axiomaticnaivemon}

\subsection{Preliminary facts}\label{preliminaries}
Here is a complete list of the axioms of independence relations, those are the same as in \cite{delbee2023axiomatic}.\\

\framebox{
\begin{minipage}{0.90\textwidth}
\textsc{Definition} (Axioms of independence relations).
\begin{enumerate}[$(1)$]
\item (\setword{finite character}{FIN}) If $a\ind_C B$ for all finite $a\seq A$, then $A\ind_C B$.
  \item (\setword{existence}{EX}) $A\ind_C C$ for any $A$ and $C$.
  \item (\setword{symmetry}{SYM}) If $A\ind_C B$ then $B\ind_C A$.
  \item (\setword{local character}{LOC}) For all $A$ there is a cardinal $\kappa = \kappa(A)$ such that for all $B$ there is $B_0\seq B$ with $\abs{B_0} <\kappa$ with $A\ind_{B_0} B$.
  \item (right \setword{normality}{NOR}) If $A\ind_C B$ then $A\ind_C BC$.
\item (right \setword{monotonicity}{MON}) If $A\ind_C BD$ then $A\ind_C B$.
\item (right \setword{base monotonicity}{BMON}) Given $C\seq B\seq D$ if $A\ind_C D$ then $A\ind_{B} D$.
  \item (right \setword{transitivity}{TRA}) Given $C\seq B\seq D$, if $A\ind_{C} B$ and $A\ind_B D$ then $A\ind_C D$.
    \item (\setword{anti-reflexivity}{AREF}) If $a\ind_C a$ then $a\in \cl(C)$;
    \item (right \setword{closure}{CLO}) $A\ind_C B \implies  A\ind _C \cl(B)$.
    \item (\setword{strong closure}{SCLO}) $A\ind_C B \iff  \cl(AC)\ind _{\cl(C)} \cl(BC)$.
  \item (\setword{strong finite character}{STRFIN}) If $a\nind_C B$ there exists $\phi(x)\in\tp(a/BC)$ such that $a'\nind _C B$ for all $a'\models \phi(x)$.
  \item (\setword{extension}{EXT}) If $A\ind_C B$ then for any $D\supseteq B$ there is $A'\equiv_{BC} A$ with $A'\ind_C D$.
  \item (\setword{full existence}{FEX}) For all $A,B,C$ there exists $A'\equiv_C A$ such that $A'\ind _C B$.
  \item (\setword{the independence theorem}{INDTHM} over models) Let $M$ be a small model, and assume $A\ind_M B$, $C_1\ind_M A$, $C_2\ind_M B$, and $C_1\equiv_M C_2$. Then there is a set $C$ such that $C\ind_M AB$, $C\equiv_{MA}C_1$, and $C\equiv_{MB}C_2$.
  \item (\setword{stationarity}{STAT} over models) Let $M$ be a small model, and assume $C_1\ind_M A$, $C_2\ind_M A$, and $C_1\equiv_M C_2$. Then $C_1\equiv_{MA} C_2$.
\end{enumerate}
\end{minipage}
}

Results in this section are classical and appear for instance in \cite{delbee2023axiomatic}. 

\begin{definition}\label{def:mo,otonisation}
    Let $\ind$ be a ternary relation and $\cl$ a closure operator. We associate the \textit{monotonisation $\indi M$ of $\ind$} which is defined as the following:
    \[A\indi M_C B\iff A\ind_D B \text{ for all $D$ with $C\seq D\seq \cl(BC)$}.\]
    We define the relation $\indi *$:
    \[A\indi * _C B \iff \text{for all $D\supseteq B$, there exists $A'\equiv_{BC} A$ with $A'\ind_C D$}\]
\end{definition}

\begin{proposition}[\cite{delbee2023axiomatic}, Proposition 1.2.14]\label{prop:monotonisation}
  The relation $\indi{M} $ satisfies right \ref{BMON}. \begin{itemize}
      \item If $\ind$ satisfies left or right \ref{MON}, left or right \ref{CLO}, left \ref{NOR}, so does $\indi M$. If $\ind$ further satisfies right \ref{NOR} or left \ref{TRA}, then so does $\indi M$. 
      \item If $\ind $ satisfies left or right \ref{CLO} then so does $\indi M$.
      \item If $\indi 0\to \ind$ and $\indi 0$ satisfies the right-sided instance of: \ref{NOR}, \ref{MON}, \ref{CLO} and \ref{BMON}, then $\indi 0\rightarrow \indi M$.
  \end{itemize}
\end{proposition}

We always have $\indi * \to \ind$. By definition, if $\indi 0 \to \ind$ and $\indi 0$ satisfies \ref{EXT}, then $\indi 0\to \indi *$.

\begin{proposition}[\cite{delbee2023axiomatic}, Proposition 4.1.17]\label{prop:forcingextpreservation}
    If $\ind$ is invariant and satisfies left and right \ref{MON} then $\indi *$ is invariant and satisfies left and right \ref{MON}, right \ref{NOR}, \ref{EXT} and right \ref{CLO}. If $\ind$ satisfies one of the following property: right \ref{BMON}, left \ref{TRA}, left \ref{NOR}, \ref{AREF} then so does $\indi *$.
\end{proposition}

\begin{theorem}[\cite{delbee2023axiomatic}, Theorem 4.1.24]\label{thm:forcingBMON+EXTforkingnaive}
    Let $\indi 0$ be an invariant relation satisfying right \ref{MON}, right \ref{BMON} and \ref{LOC} (for instance, take $\indi 0 = \indi h$).

  Let $\ind$ be an invariant relation satisfying left and right \ref{MON} and the following property ($\indi 0$-amalgamation over models):
  \begin{center}
      if $c_1\equiv_M c_2$ and $c_1\ind_M a$, $c_2\ind_M b$ and $a\indi 0 _M b$
      then there exists $c$ with $c\ind_E ab$ and $c\equiv_{Ma} c_1$, $c \equiv_{Mb} c_2$.
  \end{center}
  Then $\indi{M}^{*} \rightarrow \indi f$. 
\end{theorem}

The following is a standard use of Erd\"os-Rado and compactness.
\begin{lemma}[\cite{delbee2023axiomatic}, Lemma 3.2.18]\label{lm:LOCcharactergivesMS}
    Let $(b_i)_{i<\omega}$ be a $C$-indiscernible sequence and $\ind$ be an invariant relation satisfying right \ref{MON}, right \ref{BMON} and \ref{LOC}. Then there exists a model $M$ containing $C$ such that $(b_i)_{i<\omega}$ is an $M$-indiscernible $\ind$-Morley sequence over $M$.
\end{lemma}

\subsection{Naive monotonisation}

\begin{definition}\label{def:mo,otonisation}
    Let $\ind$ be a ternary relation. We associate the \textit{naive monotonisation $\indi{m}$ of $\ind$} which is defined as the monotonisation with respect to the trivial closure operator:
    \[A\indi m_C B\iff A\ind_D B \text{ for all $D$ with $C\seq D\seq BC$}.\]
\end{definition}

\begin{proposition}\label{prop:naivemonotonisation}
  The relation $\indi{m} $ satisfies right \ref{BMON}. 
  \begin{itemize}
      \item If $\ind$ satisfies left or right \ref{MON}, left \ref{NOR}, so does $\indi m$. If $\ind$ further satisfies right \ref{NOR} or left \ref{TRA}, then so does $\indi m$. 
      \item If $\indi 0\to \ind$ and $\indi 0$ satisfies the right-sided instance of: \ref{NOR}, \ref{MON} and \ref{BMON}, then $\indi 0\rightarrow \indi m$.
  \end{itemize}
\end{proposition}

\begin{proof}
The beginning and the first item are obtained by applying Proposition \ref{prop:monotonisation} with the trivial closure operator.

We prove the last item. If $A\indi 0 _C B$, then by \ref{NOR}, we have $A\indi 0 _C BC$. Then, by \ref{BMON}, for all $D$ with $C\seq D\seq BC$ we have $A\indi 0 _D BC$ so $A\indi 0 _D B$ by \ref{MON}. As $\indi 0\to \ind$, we get $A\ind_D B$ . We conclude that $A\indi m_C B$. 
\end{proof}

\begin{corollary}\label{cor:forcingBMONandEXT}
    If $\ind$ is invariant and satisfies left and right \ref{MON} then $\indi m ^*$ is invariant and satisfies left and right \ref{MON}, right \ref{NOR}, right \ref{CLO}, right \ref{BMON} and \ref{EXT}. Further $\indi m ^*\to \indi m\to \ind$.
\end{corollary}
\begin{proof}
    By Proposition \ref{prop:naivemonotonisation}, $\indi m$ satisfies left and right \ref{MON} and right \ref{BMON}. By Proposition \ref{prop:forcingextpreservation}, $\indi m ^*$ satisfies right \ref{MON}, right \ref{BMON} and \ref{EXT}. It is standard that \ref{EXT} implies right \ref{NOR} and right \ref{CLO}. As $\indi m ^*\to \ind$ and $\indi m ^*$ satisfies the right-sided versions of \ref{NOR}, \ref{MON} and \ref{BMON}, Proposition \ref{prop:naivemonotonisation} applies and $\indi m ^*\to \indi m$.
\end{proof}

\begin{lemma}\label{lm:twomonsamextension}
    Assume that $\ind$ satisfies left and right \ref{NOR} and \ref{MON} then $\indi M^* = \indi m ^*$.
\end{lemma}
\begin{proof}
    We have $\indi M^*\to \indi m^*$ by Propositions \ref{prop:naivemonotonisation} and \ref{prop:forcingextpreservation} because $\indi M^*$ satisfies right \ref{BMON} and \ref{EXT}
    and $\indi M^*\to \ind$ . By Corollary \ref{cor:forcingBMONandEXT} $\indi m ^*$ satisfies left and right \ref{NOR}, \ref{MON}, right \ref{CLO}, right \ref{BMON} and \ref{EXT}. As $\indi m ^*\to \ind$ and $\indi m ^*$ satisfies right \ref{CLO}, right \ref{BMON}, we get $\indi m^*\to \indi M$. As  $\indi m ^*$ satisfies \ref{EXT}, we conclude $\indi m^* \to \indi M^*$.
\end{proof}

We can already conclude the following correct version of \cite[Proposition 4.14.]{delbee21acfg}.
\begin{corollary}\label{cor:correctversionofprop4.14}
    Let $\ind$ be an invariant relation satisfying left and right \ref{MON} such that
    \begin{enumerate}
        \item $\ind$ is weaker than $\indi d$
        \item $\ind$ satisfies $\indi h$-amalgamation over models
        \item $\indi m$ satisfies \ref{EXT}
    \end{enumerate}
    then $\indi m = \indi f = \indi d$.
\end{corollary}
\begin{proof}
    As $\indi d\to \ind$ and $\indi d$ satisfies \ref{NOR}, \ref{MON} and \ref{BMON} we have $\indi d\to \indi m$ by Proposition \ref{prop:naivemonotonisation}. By Lemma \ref{lm:twomonsamextension} $\indi M ^* = \indi m ^*$, hence by Theorem \ref{thm:forcingBMON+EXTforkingnaive} we have $\indi m ^* \to \indi f$. As $\indi m$ satisfies \ref{EXT} we have $\indi m ^* = \indi m$. We conclude that $\indi f \to \indi d\to \indi m\to \indi f$ i.e. $\indi m = \indi f = \indi d$.
\end{proof}

\begin{remark}
    In Corollary \ref{cor:correctversionofprop4.14}, we may also conclude that $\indi f = \indi d = \indi m = \indi M$, simply because we always have $\indi f\to \indi M\to \indi m$. Indeed $\indi M\to \indi m$ is clear, and $\indi f\to \indi M$ follows from Propositions \ref{prop:monotonisation} and \ref{prop:forcingextpreservation} because $\indi f$ always satisfies right \ref{CLO} and right \ref{BMON}.
\end{remark}

For completeness, we include a proof of the following (which already follows from Theorem \ref{thm:forcingBMON+EXTforkingnaive} and Lemma \ref{lm:twomonsamextension}).

\begin{theorem}\label{thm:forcingBMON+EXTforkingnaive2}
    Let $\indi 0$ be an invariant relation satisfying right \ref{MON}, right \ref{BMON} and \ref{LOC} (for instance, take $\indi 0 = \indi h$).

  Let $\ind$ be an invariant relation satisfying left and right \ref{MON} and the following property ($\indi 0$-amalgamation over models):
  \begin{center}
      if $c_1\equiv_M c_2$ and $c_1\ind_M a$, $c_2\ind_M b$ and $a\indi 0 _M b$
      then there exists $c$ with $c\ind_E ab$ and $c\equiv_{Ma} c_1$, $c \equiv_{Mb} c_2$.
  \end{center}
  Then $\indi{m}^{*} \rightarrow \indi f$. 
\end{theorem}
\begin{proof}
By Corollary \ref{cor:forcingBMONandEXT}, the relation ${\indi{m}} ^{*}$ satisfies satisfies left and right \ref{MON}, right \ref{NOR}, right \ref{CLO}, right \ref{BMON} and \ref{EXT}. Note that we do not use right \ref{CLO} in this proof.

We show that $\indi m ^* \rightarrow \indi d$, the result follows from $\indi f = {\indi d}^*$. 

Assume that $a\indi m^* _C b$, for some $a,b,C$. Let $(b_i)_{i<\omega}$ be a $C$-indiscernible sequence with $b = b_0$. By Lemma \ref{lm:LOCcharactergivesMS}, there exists a model $M\supseteq C$ such that $(b_i)_{i<\omega}$ is an $M$-indiscernible $\indi 0$-Morley sequence over $M$, i.e. $b_i\indi 0_M b_{<i}$ for all $i<\omega$.

By \ref{EXT} there exists $a'$ such that $a'\equiv_{Cb} a$ and $a' \indi m ^*_{C}   b M$. It follows from \ref{BMON} and right \ref{MON} that
    $$a' \ind_{M} b.$$

For each $i\geq 0$ there exists an automorphism $\sigma_i$ over $M$ sending $b = b_0$ to $b_i$, so setting $a_i' = \sigma_i(a')$ we have:
$a_i'b_i \equiv_{M} a'b$ hence by invariance $a_i'\ind_{M} b_i$. Note that $a'b\equiv_C ab$.

\begin{claim}
    There exists $a''$ such that $a''b_i \equiv_{M} a'b$ for all $ i<\omega$.
\end{claim} 

\begin{proof}[Proof of the claim]
    By induction and compactness, it is sufficient to show that for all $i<\omega$, there exists $a_i''$ such that for all $k\leq i$ we have $a''_i b_k \equiv_{M} a' b$ and $a_i''\ind_{M} b_{\leq i}$. For the case $i = 0$ take $a''_0 = a'$. Assume that $a_i''$ has been constructed, we have 
$$ a_{i+1}' \ind_{M}b_{i+1}\text{ and } b_{ i+1} \indi{0}_{M} b_{\leq i}  \text{ and }  a_i''\ind_{M}b_{\leq i}.$$

As $a_{i+1}' \equiv_{M} a_i''$, by $\indi{0}$-amalgamation over models, there exists $a_{i+1}''$ such that 
\begin{enumerate}
    \item $a_{i+1}''b_{i+1} \equiv_{M} a_{i+1}'b_{i+1}$
    \item $a_{i+1}'' b_{\leq i}\equiv_{M} a_i''b_{\leq i}$
    \item $a_{i+1}'' \ind_{M} b_{\leq i+1}$.
\end{enumerate}
By induction and compactness, there exists $a''$ such that $a''b_i \equiv_{M} a b$ for all $i<\omega$, which proves the claim.
\end{proof}

Let $a''$ be as in the claim, then as $a'b\equiv_C ab$ we have $a''b_i \equiv_C ab$ for all $i<\omega$, hence $a\indi{d}_C b$.
\end{proof}

Another consequence of Theorem \ref{thm:forcingBMON+EXTforkingnaive2} (or Lemma \ref{lm:twomonsamextension}) is the following, which implies that either version of the monotonisation are enough to get a description of forking in NSOP$_1$ theories.
\begin{proposition}\label{prop:twomonsameextenNSOP1}
    Let $T$ be an NSOP$_1$ theory and $\indi K$ be Kim-independence in $T$. Then
    \[\indi f = (\indi K\ )^{M*} = (\indi K\ )^{m*}\]
\end{proposition}
\begin{proof}
    We already have $\indi f \to (\indi K)^{M*} \to (\indi K\ )^{m*}$. In an NSOP$_1$ theory, $\indi K$ satisfies the hypothesis of Theorem \ref{thm:forcingBMON+EXTforkingnaive2} with $\ind = \indi 0 = \indi K$ thus $(\indi K\ )^{m*}\to \indi f$.
\end{proof}

\begin{remark}
    In ACFG, we have that $\indi K^{M*} = \indi K ^M = \indi f$ and $\indi K^m\neq \indi K^M$ (otherwise $\indi K^M\to \indi d\to \indi K^m$ would yield $\indi f = \indi d)$). In particular $\indi K ^m \neq \indi K ^{m*}$ and $\indi K ^m$ fails \ref{EXT}. Note that in ACFG, we have $\indi K ^{m*} = \indi K ^{M*}$ but $\indi K ^m\neq \indi K ^M$.
\end{remark}

\subsection{Forcing algebraicity}

Now we expand a bit on ideas from \cite{conant2023surprising}.

\begin{definition}\label{def:rightclosure}
    Let $\ind$ be a ternary relation. We associate the \textit{(right) closure extension $\indi{c}$ of $\ind$} which is defined as :
    \[A\indi{c}_C B\iff A\ind_C \acl(BC).\]
\end{definition}

\begin{remark}
    In \cite{conant2023surprising}, we have $\indi d ^c = \indi {da}$. 
\end{remark}

We naturally have the following:
\begin{proposition}\label{prop:forcingclosure}
    Let $\ind$ be an invariant relation. Then $\indi c$ satisfies right \ref{CLO} and right \ref{NOR}. If $\ind$ satisfies right \ref{MON}, then $\indi c\to \ind$. \begin{itemize}
        \item If $\ind$ satisfies left or right \ref{MON}, left \ref{NOR}, so does $\indi c$.
        \item If $\indi 0\to \ind$ and $\indi 0$ satisfies right \ref{NOR} and right \ref{CLO}, then $\indi 0 \to \indi c$.
        \item If $\ind$ satisfies right \ref{BMON} then so does $\indi c$.
    \end{itemize}
\end{proposition}
\begin{proof}
    For right \ref{CLO}: if $A\indi c_C B$ we have $A\ind_C \acl(BC)$. As $\acl(\acl(BC)) = \acl(BC)$ we have $A\ind_C \acl(\acl(BC))$ i.e. $A\indi c _C \acl(BC)$. For right \ref{NOR}: if $A\indi c_C B$ then $A\ind_C \acl(BC)$ so $A\indi c_C BC$. If $\ind$ satisfies right \ref{MON}, then $A\indi c_C B$ implies $A\ind_C \acl(BC)$ which implies $A\ind_C B$.

    For the first item: the left-sided properties are trivial. If $\ind$ satisfies right \ref{MON}, then $A\indi c_C BD$ implies $A\ind_C \acl(CBD)$ which implies $A\ind_C \acl(BC)$ so $A\indi c _C B$. 

    For the second item, if $A\indi 0_C B$ then by right \ref{NOR} and right \ref{CLO} we have $A\indi 0_C \acl(BC)$ which implies $A\ind_C \acl(BC)$ so $A\indi c_C B$.

    For the third item, assume that $A\indi c _C B$ and $C\seq D\seq B$. Then $A\ind_C \acl(B)$ hence by \ref{BMON} $A\ind_D \acl(B)$ so $A\indi c _D B$.
\end{proof}

\begin{proposition}\label{prop:naivetomonviac}
    Assume that $\ind$ satisfies right \ref{NOR} and right \ref{MON} then 
    \[(\indi m)^c\to \indi M \]
    If $\ind$ further satisfies right \ref{CLO} then $(\indi m)^c = \indi M$.
\end{proposition}
\begin{proof}
    First, $\indi m$ satisfies the right-sided version of \ref{NOR}, \ref{MON}, \ref{BMON}, by Proposition \ref{prop:naivemonotonisation}. By Proposition \ref{prop:forcingclosure} we have $\indi m ^c$ satisfies the right-sided version of \ref{NOR}, \ref{MON}, \ref{BMON} and \ref{CLO}. Again by Proposition \ref{prop:forcingclosure} $\indi m ^c\to \ind$ hence by Proposition \ref{prop:monotonisation} we have $\indi m ^c \to \indi M$.

    Assume that $\ind$ satisfies right \ref{CLO}, then by Proposition \ref{prop:monotonisation}, $\indi M$ satisfies the right-sided instances of \ref{NOR} \ref{MON}, \ref{BMON} and \ref{CLO}. As $\indi M\to \ind$, from Proposition \ref{prop:naivemonotonisation} we have $\indi M\to \indi m$. From $\indi M\to \indi m$ and Proposition \ref{prop:forcingclosure} we get $\indi M\to \indi m ^c$.
\end{proof}

The essential difference between $\indi M $ and $\indi m$ is that the former preserves right \ref{CLO} and the other one does not a priori. However this distinction does not appear in a pregeometric theory:
\begin{fact}[\cite{conant2023surprising}, Remark 2.19]\label{fact:monpregeom}
    If $T$ is pregeometric then $\indi a ^M = \indi a ^m$.
\end{fact}

\begin{remark}\label{rk:pregeommon}
In ACFG, we have $\indi K\to \indi \alg\ $. As in ACFG, the algebraic closure coincide with the field theoretic algebraic closure (see \cite[Section 2.1]{delbee21acfg}), we have that $\indi \alg = (\indi a)^m = (\indi a)^M$. In particular:\[\indi d\to \indi K\to (\indi a)^M.\]
\end{remark}

\bibliographystyle{plain}
\bibliography{biblio}

\end{document}

%% file: preamble.tex
\usepackage{anysize}
\marginsize{3cm}{3cm}{3cm}{3cm}

\usepackage{amssymb,latexsym}
\usepackage{amsmath,amsthm}
\usepackage{amsfonts,mathrsfs}
\usepackage{mathtools}

\usepackage{tikz-cd}

\usepackage{todonotes}

\usepackage[all]{xy}
\usepackage{graphicx}

\usepackage{xcolor,url}
\usepackage{hyperref}
\hypersetup{
    colorlinks,
    citecolor=blue,
    filecolor=blue,
    linkcolor=black,
    urlcolor=blue
}

\usepackage[shortlabels]{enumitem}

\usepackage{halloweenmath}

\theoremstyle{plain}
\newtheorem{theorem}{Theorem}[section]

\newtheorem{corollary}[theorem]{Corollary}
\newtheorem{lemma}[theorem]{Lemma}
\newtheorem{proposition}[theorem]{Proposition}
\newtheorem{fact}[theorem]{Fact}

\newtheorem*{theorem*}{Theorem}

\theoremstyle{definition}
\newtheorem{definition}[theorem]{Definition}

\theoremstyle{remark}
\newtheorem{remark}[theorem]{Remark}

\newtheorem{claim}{Claim}

\newcommand\F{\mathbb{F}}

\newcommand\LL{\mathscr{L}}

\newcommand{\acl}{\mathrm{acl}}

\newcommand{\cl}{\mathrm{cl}}

\newcommand{\alg}{\mathrm{alg}}

\newcommand{\ac}{\mathrm{ac}}
\newcommand{\eq}{\mathrm{eq}}

\newcommand{\tp}{\mathrm{tp}}

\def\seq{\subseteq}

\newcommand{\set}[1]{\left\{ {#1} \right\}}

\newcommand{\abs}[1]{\lvert {#1} \rvert}

\DeclareMathOperator{\Span}{Span}

\definecolor{airforceblue}{rgb}{0.36, 0.54, 0.66}

\def\IND{\setbox0=\hbox{$x$}\kern\wd0\hbox to 0pt{\hss$\mid$\hss}
\lower.9\ht0\hbox to 0pt{\hss$\smile$\hss}\kern\wd0}
\def\NotIND{\setbox0=\hbox{$x$}\kern\wd0\hbox to 0pt{\mathchardef\nn=12854\hss$\nn$\kern1.4\wd0\hss}\hbox to 0pt{\hss$\mid$\hss}\lower.9\ht0 \hbox to 0pt{\hss$\smile$\hss}\kern\wd0}

\def\ind{\mathop{\mathpalette\IND{}}}
\def\nind{\mathop{\mathpalette\NotIND{}}}

\def\indi#1{\mathop{\mathpalette\IND{}^{\!\!\!\!\rlap{$\scriptstyle{#1}$}\,\,\,\,}}}

\def\nindi#1{\mathop{\mathpalette\NotIND{}^{\!\!\!\rlap{$\scriptstyle{#1}$}\,\,\,}}}

\makeatletter
\newcommand{\setword}[2]{%
  \phantomsection
  #1\def\@currentlabel{\unexpanded{#1}}\label{#2}%
}
\makeatother

\renewcommand{\models}{\vDash}

